\documentclass[a4paper,12pt]{article}
\usepackage[T1]{fontenc}
\usepackage[utf8]{inputenc}
\usepackage[english]{babel}
\usepackage{amsmath,amsthm,amssymb,latexsym,amsfonts, mathtools}
\usepackage{mathrsfs} 
\usepackage{tikz-cd}
\usepackage{listings}
\usepackage{faktor} 
\usepackage{nicefrac}
\usepackage{color,soul}
\usepackage[hidelinks]{hyperref}

\renewcommand{\epsilon}{\varepsilon}
\renewcommand{\theta}{\vartheta}
\renewcommand{\rho}{\varrho}
\let\temp\phi
\let\phi\varphi
\let\varphi\temp

\newcommand{\Z}{\mathbb{Z}}
\newcommand{\Q}{\mathbb{Q}}
\newcommand{\R}{\mathbb{R}}
\newcommand{\C}{\mathbb{C}}

\newcommand{\trdeg}{\text{trdeg}}

\theoremstyle{definition}

\theoremstyle{plain}
\newtheorem{proposition}{Proposition}
\theoremstyle{plain}
\newtheorem{theorem}{Theorem}
\theoremstyle{plain}
\newtheorem{lemma}{Lemma}
\theoremstyle{plain}
\newtheorem{corollary}{Corollary}
\theoremstyle{plain}

\theoremstyle{definition}
\newtheorem{lemma*}{Lemma}
\theoremstyle{definition}

\begin{document}
	\title{On the algebraic independence of periods of abelian varieties and their exponentials}
	\author{Riccardo Tosi}
	\date{September 6, 2023}
	\maketitle
	
    \begin{abstract}
    We generalize a result by Vasil'ev on the algebraic independence of periods of abelian varieties to the case when some of these periods are replaced by their exponentials. We eventually derive some applications to values of the beta function at rational points.
    \end{abstract}
 
	\section{Introduction}

    Let X be an abelian variety of dimension $g$ defined over a number field $K$. It is possible to choose a basis $\phi_1,\dots,\phi_{2g}$ of $\text{H}^1_{\text{dR}}(X,K)$ in such a way that $\phi_1,\dots,\phi_g\in \Gamma(X,\Omega^1_{X/K})$. Fix $2g$ paths $\gamma_1,\dots,\gamma_{2g}$ on $X(\C)$ which induce a basis for $\text{H}_1(X(\C),\Q)$ and consider the period matrix 
    \[
	\Omega=(\omega_{ij})=\left( \int_{\gamma_j} \phi_i\right), \quad i,j=1,\dots,2g.
	\] 
    
	As far as the transcendence properties of the periods of $X$ are concerned, Chudnovsky \cite{Chudnovsky-main} proved that there are always two algebraically independent numbers among the entries of $\Omega$, following the methods introduced in his celebrated theorem on the transcendence degree of periods of elliptic curves. Later on, Vasil'ev \cite{Vasil'ev} managed to show that two algebraically independent periods exist even when restricting to any $g+1$ rows of $\Omega$. 
 
    A bunch of applications have been derived, in particular concerning values of the beta and gamma function at rational points. To quote some, Chudnovsky's result applied to the case of elliptic curves with complex multiplication led to the first proof of the algebraic independence of $\pi$ with both $\Gamma(\frac{1}{4})$ and $\Gamma(\frac{1}{3})$. On the other hand, Vasil'ev's refinement yields for any $n\ge 3$ the existence of at least two algebraically independent numbers among $\pi, \Gamma\left(\frac{1}{n}\right), \dots, \Gamma\left(\frac{g}{n}\right)$, where $g=\lfloor \frac{n-1}{2}\rfloor$.
    
	The aim of the present work is to expose a transcendence result along these lines in the case when some periods are replaced by their exponentials. Before stating our main result, we will briefly recall some well-known facts regarding transcendence measures.
 
	If $\omega\in\C$ is transcendental, we say that a function $\phi:\R^2\to \R$ is a \emph{transcendence measure} for $\omega$ if for all non-zero polynomials $P$ with integer coefficients of degree at most $n$ and height at most $h$ we have $\log |P(\omega)|\ge -\phi(n,\log h)$. A real number $\tau>0$ is said to be a \emph{transcendence type} for $\omega$ if there is a constant $c(\omega,\tau)>0$ such that $c(\omega,\tau)(n+\log h)^{\tau}$ is a transcendence measure for $\omega$.
 
	An elementary argument involving the pigeonhole principle shows that any transcendence type $\tau$ for a transcendental number $\omega$ must satisfy $\tau\ge 2$. On the other hand, almost all transcendental numbers admit a transcendence type $\le 2+\epsilon$ for all $\epsilon>0$; we refer for instance to \cite{Amoroso} for a proof of this assertion. The only particular number which is known to have a transcendence type $\le 2+\epsilon$ for all $\epsilon>0$ is $\pi$, for which explicit transcendence measures can be found in \cite{Waldschmidt-measures}. 
 
	We may now state the main result that we wish to prove. Fix $g$ non-zero complex numbers $\xi_1,\dots, \xi_g$ and define the $g\times 2g$ matrix $E$ whose $i$-th row, for $i=1,\dots, g$, coincides with
	\[
	\left(e^{\xi_i \omega_{i1}}, \dots, e^{\xi_i \omega_{i,2g}}\right).
	\]
	Pick two integers $m,n$ such that $1\le m \le 2g$ and $0\le n\le g$. Let us choose $m$ rows of the matrix $\Omega$ and $n$ rows of the matrix $E$, say the rows $i_1,\dots, i_n$ for $1\le i_1,\dots, i_n\le g$. Let $S$ be the set made up of the entries of the chosen rows, together with $\xi_{i_1},\dots,\xi_{i_n}$. According to \cite{Schneider}, each row of $\Omega$ cannot have only algebraic entries, so the field $\Q(S)$ has transcendence degree at least $1$ over $\Q$.
	\begin{theorem}
		\label{theo: 3}
		Suppose that $\Q(S)$ has transcendence degree exactly $1$ over $\Q$. If $2m+n>2g$, then any transcendence type $\tau$ of a transcendental number $\omega\in\Q(S)$ satisfies 
		\[
		\tau\ge2+\frac{2m+n-2g}{2g+n}.
		\]
	\end{theorem}

	In particular, any such transcendence type is bounded away from $2$, which is the case for almost no transcendental number. In fact, we expect the transcendence degree of $\Q(S)$ to be at least $2$ whenever $2m+n>2g$. We will later expose the main obstructions to obtaining a result of this kind.
 
    For the proof of Theorem~\ref{theo: 3}, we follow closely the strategy of \cite{Vasil'ev}, which can be recovered from the case $n=0$. It relies on the construction of an auxiliary function vanishing with high multiplicity on prescribed points of a lattice in $\C^g$ by exploiting Siegel's Lemma. As far as the structure of our exposition is concerned, the next section covers the general setting and some preliminary lemmas. The third section revolves around the construction of the aforementioned auxiliary function, whereas the conclusion of the proof of Theorem~\ref{theo: 3} takes place in the fourth section, and is carried out by analytic means via a generalization by Bombieri and Lang of Schwarz's Lemma to several variables. We will eventually reserve the last section for some applications to values of the beta function at rational points, in which it is possible to turn Theorem~\ref{theo: 3} into an algebraic independence criterion. For example, we shall prove that there are at least two algebraically independent numbers among
    \[
    B\left( \frac{1}{12},\frac{1}{12}\right), \; B\left(\frac{5}{12},\frac{5}{12}\right),\; \pi, \; e^{\pi^2},\; e^{i\pi^2}.
    \]

    \subsection*{Acknowledgements} I wish to thank Johannes Sprang for his continuous support, his patience and his thorough dedication in helping me write the present text. All this would not have been possible without him.
    
    While writing this article, I was supported by the DFG Research Training Group 2553 \textit{Symmetries and classifying spaces: analytic, arithmetic and derived}.
    
	\section{Notation and setting}
	We first introduce some notation. The standard coordinates in $\C^g$ will be denoted by $z=(z_1,\dots,z_g)\in\C^g$. For $j=1,\dots, 2g$, we set $\lambda_j= (\omega_{1j},\dots,\omega_{gj})$, that is, $\lambda_j$ is the vector of the first $g$ entries of the $j$-th column of $\Omega$. Since we have chosen $\phi_1,\dots,\phi_g\in \Gamma(X,\Omega^1_{X/K})$, the vectors $\lambda_1,\dots,\lambda_{2g}$ generate a lattice $\Lambda$ of $\C^g$ which yields an isomorphism $X(\C)\cong \C^g/\Lambda$. 
 
    For an integral vector $k=(k_1,\dots, k_g)\in \Z^g$, we set
	\[
	|k|\coloneqq \sum_{j=1}^g |k_j|, \quad \|k\|\coloneqq \max_{j=1,\dots,g} |k_j|, \quad k\cdot \lambda\coloneqq \sum_{j=1}^g k_j\lambda_j, \quad z^k=z_1^{k_1}\dots z_g^{k_g}.
	\]
	Given a differentiation operator
	\[
	\partial=\frac{\partial^m}{\partial z_1^{m_1}\dots\partial z_g^{m_g}},
	\] 
	its order will be denoted by $|\partial|=m=m_1+\dots +m_g$. The notation $|\partial|=0$ signifies that $\partial$ is the identity operator. For an entire function $f:\C^g\to \C$, we write $|f|_R$ for the maximum modulus of $f$ in the closed ball of radius $R\ge 0$ centred at the origin.
 
	Given a polynomial $P\in \C[x_1,\dots,x_n,y_1,\dots,y_m]$, we set for the sake of brevity $x=(x_1,\dots,x_n)$, $y=(y_1,\dots,y_m)$. We will write $\deg P$ for the degree of $P$, while the notation $\deg_x P$ will refer to the degree of $P$ in $x$, and similarly for $\deg_y P$. The maximum modulus of the coefficients of $P$, the \emph{height} of $P$, will be denoted by $H(P)$. If $P\neq 0$, we also denote by $t(P)$ the \emph{type} of $P$, which stands for the maximum between the degree and the logarithm of the height of $P$. 
 
	Let now $L$ be a subfield of $\C$ such that the extension $L/\Q$ is finitely generated. We may write $L=\Q(\omega_1,\dots,\omega_q,\chi)$ for some $\omega_1,\dots,\omega_q\in\C$ algebraically independent and $\chi\in \C$ integral over $\Z[\omega_1,\dots,\omega_q]$ of degree $d$, say. Any element $a\in L$ can be written uniquely in the form $aP=\sum_{i=1}^d P_i y^{i-1}$ for suitable $P,P_1,\dots, P_d\in \Z[\omega_1,\dots,\omega_q]$ such that $\gcd(P,P_1,\dots, P_d)=1$. For $i=1,\dots, q$, we define the degree of $a$ in $\omega_i$ and the type of $a$ respectively as
	\begin{gather*}
	\deg_{\omega_i} a \coloneqq \max \{\deg_{\omega_i} P, \deg_{\omega_i} P_1,\dots, \deg_{\omega_i} P_d \},\\
	 t(a)\coloneqq \max\{t(P),t(P_1),\dots, t(P_d) \}.
	\end{gather*}
	We remark that these notions depend on the choice of generators $\omega_1,\dots, \omega_q,\chi$ for the extension $L/\Q$. Upper bounds for the type of a number in $L$ can be reduced to computations with types of polynomials via the following
    \begin{lemma}
		\label{lemma: how to deal with type proved}
		There is a constant $c>0$, only depending on $\omega_1,\dots,\omega_q,\chi$, such that for any polynomial $P\in\Z[x_1,\dots,x_q,y]$
		\[
		t(P(\omega_1,\dots,\omega_q,\chi))\le c \,t(P)+c\log(\deg P).
		\]
    \end{lemma}
    \begin{proof}
    \cite[Lemme 4.2.5]{Waldschimdt-fr}.
    \end{proof}
    
	Let us now consider the matrix $\Omega$ and the lattice $\Lambda$ defined above. There exist $3g$ meromorphic functions $A_1,\dots, A_g,H_1,\dots, H_{2g}:\C^g\to \C$ with the following properties:
	\begin{enumerate}
		\item all these functions vanish at the origin of $\C^g$;
		\item for all $i=1,\dots, g$, $j=1,\dots, 2g$ we have $A_i(z+\lambda_j)=A_i(z)$;
		\item for all $i,j=1,\dots, 2g$ we have $H_i(z+\lambda_j)=H_i(z)+\omega_{ij}$.
	\end{enumerate}
	The functions $A_1,\dots, A_g$ are called \emph{abelian}, while $H_1,\dots, H_{2g}$ \emph{quasi-periodic}. Moreover, there exists a theta function $\theta:\C^g\to \C$ which is entire, does not vanish at the origin and is a denominator for the abelian and quasi-periodic functions, in the sense that $\theta A_1,\dots, \theta A_g,\theta H_1,\dots, \theta H_{2g}$ are entire. We recall that, by definition of theta function, for all $j=1,\dots,2g$ we may find $u_{j1},\dots, u_{jg}, v_j\in \C$ satisfying
	\[
	\theta(z+\lambda_j)=\theta(z)\exp( 2\pi i(u_{j1}z_1+\dots+u_{jg}z_g+v_j)).
	\]
    The proof of these facts can be found for instance in \cite{Lange-Birkenhake} and \cite{Grinspan}. We conclude this section with some arithmetic properties of these functions.
	\begin{lemma}
		\label{lemma: alg coeff series of qper funct} 
		The functions $A_1,\dots, A_g$ and $H_1,\dots, H_{2g}$ can be expanded at the origin in power series of the form 
		\[
		\sum_{|\mu|\ge 1} b_{\mu} z_1^{\mu_1}\dots z_g^{\mu_g},
		\]
		where $z=(z_1,\dots,z_g)\in \C^g$, $\mu=(\mu_1,\dots,\mu_g)\in\Z^g$ and $\mu_i\ge 0$ for all $i=1,\dots, g$. Furthermore, the coefficients $b_\mu$ lie in $K$ and enjoy the following properties:
		\begin{enumerate}
			\item There exists a positive integer $c_1$ such that the maximum modulus of the conjugates of $b_\mu$ satisfies $\le e^{c_1|\mu|}$;
			\item there exists a positive integer $c_2$ such that
			\[
			(3|\mu|)!c_2^{|\mu|}b_{\mu}
			\]
			is an algebraic integer.
		\end{enumerate}
	\end{lemma}
	\begin{proof} \cite[Corollary 4.5]{Grinspan}.
	\end{proof}

	\section{The auxiliary function}
	The first step in proving Theorem~\ref{theo: 3}, which will be carried out in this section, is the construction of an auxiliary function which vanishes with high multiplicity on a prescribed lattice in $\C^g$. This purpose will be accomplished by means of the following version of Siegel's Lemma:
	\begin{lemma}
		\label{lemma: Siegel} 
		Let $L=\Q(x_1,\dots, x_q,y)$ for some $x_1,\dots, x_q\in\C$ algebraically independent and $y$ integral over $\Z[x_1,\dots, x_q]$. Then there exists a constant $C>0$ which enjoys the following property. \\
		Let $n$ and $r$ be positive integers with $n\ge 2r$ and consider $a_{ij}\in\Z[x_1,\dots,x_q,y]$, for $i=1,\dots, n$, $j=1,\dots, r$. Then there exist $\xi_1,\dots,\xi_n\in\Z[x_1,\dots,x_q,y]$, not all zero, such that for all $j=1,\dots, r$ 
		\[
		\sum_{i=1}^n \xi_i a_{ij}=0 \quad \text{and} \quad \max_{i=1,\dots, n} t(\xi_i)\le C\left(\max_{i,j} t(a_{ij}) + \log n \right).
		\]
	\end{lemma}
	\begin{proof}
		\cite[Lemme 4.3.1]{Waldschimdt-fr}
	\end{proof}

	Without loss of generality, we may suppose that the chosen rows are the first $m$ of $\Omega$ and the first $n$ of $E$. Let $K$ be the number field, embedded in $\C$, over which the abelian variety $X$ is defined; we may suppose that $K$ is a Galois extension of $\Q$. Also, from Lemma~\ref{lemma: alg coeff series of qper funct} it follows that $K$ contains all the coefficients of the Taylor expansion of $H_1,\dots, H_m$ and $A_1,\dots,A_g$ around the origin. Set $\delta\coloneqq [K:\Q]$ and let $\alpha_1,\dots,\alpha_\delta$ be an integral basis for the ring of integers of $K$. We fix a transcendental number $\omega\in\Q(S)$ and we suppose that all numbers in $S$ are algebraic over $\Q(\omega)$. Let $\chi\in\C$ be integral over $\Z[\omega]$ and such that $\Q(\omega,\chi)$ contains both $\Q(S)$ and $K$. We let $N$ signify a sufficiently large positive integer and we write $c_1,c_2,\dots$ for positive constants depending only on $S$.
 
	We set for short $b=2m+n-2g$, which is positive by hypothesis, and we let $a$ be a real number such that 
	\[
	a>\frac{2g+n}{2m+n-2g}
	\]
	We then choose a real number $\epsilon_1$ satisfying
	\[
	0<\epsilon_1 < \min\left\{\frac{b}{2m+n}, \, \frac{ab-2g-n}{a(m+1)+g(a+1)} \right\}.
	\]
	For $\epsilon_1$ in this range of values, it is always possible to find $\epsilon_2>0$ such that
	\[
	\epsilon_1<\epsilon_2< \min\left\{\frac{b-\epsilon_1 g}{m}, \,\frac{ab-2g-n + \epsilon_1(a(m+n-g)-g )}{a(2m+n)+m} \right\}.
	\]
	We define the quantities
	\begin{align*}
	r&= m+n,\\
	t &= 2g+\epsilon_1(g+n+m)-\epsilon_2 n+n, \\
	d &= t-r\epsilon_1,\\
	d_0 &= t-r(\epsilon_1+1-\epsilon_2),
	\end{align*}
	which have been chosen in such a manner that they satisfy
	\[
	\begin{cases}
	gt+2gr=nd_0+(g+m)d,\\
	0<d_0<d<t<d_0+r,\\
	\left(2+\frac{1}{a}\right)d_0+\frac{1}{a}r<t.
	\end{cases}
	\]
	We finally set
	\[
	R=N^r, \quad T=N^t,\quad D=N^d, \quad D_0=\lfloor N^{d_0}\log N\rfloor.
	\]
	We are ready to introduce our auxiliary function.
	\begin{proposition}
		\label{prop: aux. function 3}
		There exists a constant $C(\omega)$ only depending on $S$, $\omega$, $\alpha_1,\dots,\alpha_\delta$ and $\chi$ and there exist numbers $E_{hl\nu}(\omega)\in\Z[\omega]$ not all zero satisfying the following property. Consider the function
		\[
		\Phi(z)\coloneqq C(N)^{\delta}\sum_{\|h\|\le \delta D_0}\sum_{\|l\|\le \delta D}\sum_{\|\nu\|\le \delta D} E_{hl\nu}(\omega) \prod_{r=1}^ne^{h_r\xi_r z_r}\prod_{i=1}^{m}H_{i}^{l_{i}}(z)
		\prod_{s=1}^g A_s^{\nu_s}(z),
		\]
		with 
		\[
		C(N)=C(\omega)^{nT+2gnD_0R+2gmD+1}(3T)!c_2^T ,
		\] 
		$c_2$ as in Lemma~\ref{lemma: alg coeff series of qper funct}, $h=(h_1,\dots,h_n)\in\Z^{n}_{\ge 0}$, $l=(l_1,\dots,l_m)\in\Z^{m}_{\ge0}$ and $\nu=(\nu_1,\dots,\nu_g)\in\Z^g_{\ge 0}$. Then $\Phi$ satisfies the conditions
		\[
		\partial \Phi(k\cdot \lambda)=0 \quad \text{for $|\partial|\le T$ and $k\in\Z^{2g}$ with $\|k\|\le R$}. 
		\]
		Moreover, the $E_{hl\nu}(\omega)$'s, as polynomials in $\omega$, satisfy
		\[
		t(E_{hl\nu}) \le c_3 D_0R\log (D_0R).
		\]
	\end{proposition}

	The proof of Proposition~\ref{prop: aux. function 3} relies on regarding (a minor modification of) the equations $\partial \Phi(k\cdot \lambda)=0$ as a homogeneous linear system in the $E_{hl\nu}(\omega)$'s, so as to apply Siegel's Lemma~\ref{lemma: Siegel}. The following Lemma deals with the quantities which arise as the coefficients of such linear system.
	\begin{lemma}
		\label{lemma: coeff for Siegel 3}
		For $k=(k_1,\dots,k_{2g})\in\Z^g$ and $h,l, \nu$ as above with $\|h\|\le D_0$ and $\|l\|,\|\nu\|\le D$, define
		\[
		F_{khl\nu}(z)\coloneqq \prod_{r=1}^n e^{h_r\xi_r(z_r+(k\cdot\lambda)_r)}\prod_{i=1}^{m} H_i^{l_i}(z+k\cdot \lambda) \prod_{s=1}^g A_s^{\nu_s}(z),
		\]
		and let $\partial$ be a differential operator of order $|\partial|\le T$. Then there is a constant $C(\omega)\in\Z[\omega]$ depending only on $S$, $\omega$, $\alpha_1,\dots,\alpha_\delta$ and $\chi$ such that 
		\[
		P_{\partial khl\nu}(\omega,\chi)\coloneqq C(N)\partial F_{khl\nu}(0)\in \Z[\omega,\chi],
		\]
		with $C(N)$ defined as in Proposition~\ref{prop: aux. function 3}, satisfies
		\[
		t(P_{\partial khl\nu}(\omega,\chi))\le c_4 D_0R\log (D_0R).
		\]
	\end{lemma}
	\begin{proof}
		By the quasi-periodicity of $H_i$, we have 
		\begin{align*}
		H_i^{l_i}(z+k\cdot \lambda) & =\sum_{m_0+\dots+m_{2g}=l_i}\binom{l_i}{m_0;\dots; m_{2g}} H_i^{m_0}(z)\prod_{j=1}^{2g}(k_j\omega_{ij})^{m_j}\\
		& =\sum_{m_1+\dots+m_{2g}\le l_i} \phi_{i,m_1,\dots,m_{2g}}(z)\prod_{j=1}^{2g}\omega_{ij}^{m_j},
		\end{align*}
		setting for short
		\[
		\phi_{i,m_1,\dots,m_{2g}}(z)=\frac{l_i!H_i^{l_i-m_1-\dots-m_{2g}}(z)}{m_1!\dots m_{2g}!(l_i-m_1-\dots- m_{2g})!}\prod_{j=1}^{2g}k_j^{m_j}.
		\]
		Let now $\kappa$ range through the $(g+1)\times 2g$ matrices with non-negative integer coefficients $\kappa_{ij}$ such that for $i=1,\dots,g+1$ we have $\kappa_{i1}+\dots+\kappa_{i,2g}\le l_i$. We then define $\psi_\kappa$ in such a way that
		\begin{align*}
		\prod_{i=1}^{m}H_i^{l_i}(z+k\cdot \lambda) = \sum_{\kappa} \psi_\kappa(z)\prod_{i=1}^{m}\prod_{j=1}^{2g}\omega_{ij}^{\kappa_{ij}}.
		\end{align*}
		Let us also write
		\[
		\psi_\kappa'(z)=\psi_\kappa(z)\prod_{s=1}^g A_s^{\nu_s}(z),
		\]
		with Taylor expansion around $0$ given by
		\[
		\psi_\kappa'(z)=\sum_q \beta_{\kappa, q} \frac{z_1^{q_1}\dots z_g^{q_g}}{q_1!\dots q_g!}.
		\]
		for some algebraic integers $\beta_{\kappa,q}\in K$. By \cite[Lemma 2]{Vasil'ev}, the components of $\beta_{\kappa,q}$ with respect to the integral basis $\alpha_1,\dots,\alpha_\delta$ have modulus $\le e^{c_5T\log T}$.\\
		Furthermore, we have
		\begin{align*}
		\prod_{r=1}^n e^{h_r\xi_r(z_r+(k\cdot\lambda)_r)}&=\prod_{r=1}^n e^{h_r\xi_r z_r} \prod_{r=1}^n\prod_{j=1}^{2g} \left(e^{\xi_r\omega_{rj}}\right)^{h_rk_j}\\
		&= \prod_{r=1}^n\prod_{j=1}^{2g} \left(e^{\xi_r\omega_{rj}}\right)^{h_rk_j} \sum_p \gamma_p \frac{z_1^{p_1}\dots z_g^{p_g}}{p_1!\dots p_g!},
		\end{align*}
		where $p=(p_1,\dots,p_g)\in\Z^g_{\ge 0}$ and $\gamma_p=0$ if $p_j\neq 0$ for some $j=n+1,\dots,g$, while $\gamma_p= \prod_{r=1}^n h_r^{p_r}\xi_r^{p_r}$ otherwise. It follows that
		\begin{gather*}
		\prod_{r=1}^n e^{h_r\xi_rz_r}\prod_{i=1}^{m} H_i^{l_i}(z+k\cdot \lambda) \prod_{s=1}^g A_s^{\nu_s}(z)=\\
		=\left(\sum_p \gamma_p \frac{z_1^{p_1}\dots z_g^{p_g}}{p_1!\dots p_g!}\right)\left( \sum_\kappa\sum_q \beta_{\kappa, q}\prod_{i=1}^{m}\prod_{j=1}^{2g}\omega_{ij}^{\kappa_{ij}} \frac{z_1^{q_1}\dots z_g^{q_g}}{q_1!\dots q_g!}\right)=\\
		= \sum_{p,q} \sum_\kappa  \gamma_p \beta_{\kappa,q} \prod_{i=1}^{m}\prod_{j=1}^{2g}\omega_{ij}^{\kappa_{ij}} \frac{z_1^{p_1+q_1}\dots z_g^{p_g+q_g}}{p_1!q_1!\dots p_g!q_g!}=\\
		= \sum_{p,q} \sum_\kappa  \gamma_p \beta_{\kappa,q} \binom{p_1+q_1}{p_1}\dots \binom{p_g+q_g}{p_g}\prod_{i=1}^{m}\prod_{j=1}^{2g}\omega_{ij}^{\kappa_{ij}} \frac{z_1^{p_1+q_1}\dots z_g^{p_g+q_g}}{(p_1+q_1)!\dots (p_g+q_g)!}.
		\end{gather*}
		The $t$-th term in the Taylor expansion of $F_{khl\nu}$ at $0$ therefore coincides with
		\[
		\sum_{p+q=t}\sum_\kappa  \beta_{\kappa,q} \prod_{a=1}^g\binom{p_a+q_a}{p_a}\prod_{r=1}^n h_r^{p_r}\xi_r^{p_r}\prod_{i=1}^{m}\prod_{j=1}^{2g}\omega_{ij}^{\kappa_{ij}}\prod_{r=1}^n\prod_{j=1}^{2g} \left(e^{\xi_r\omega_{rj}}\right)^{h_rk_j} \frac{z_1^{t_1}\dots z_g^{t_g}}{t_1!\dots t_g!}
		\]
		Let us now find a suitable quantity that clears out the denominator of the coefficient of the latter term. Let $C(\omega)$ be a denominator in $\Z[\omega,\chi]$ for $\alpha_1,\dots,\alpha_\delta$ and for the numbers in $S$. By \cite[Lemma 7]{Vasil'ev}, $C(\omega)(3T)!c_2^T$ is a denominator for all the $\beta_{\kappa,q}$'s. Hence
		\[
		P_{\partial khl\nu}(\omega,\chi)=C(\omega)^{nT+2gmD+2gnD_0R+1}(3T)!c_2^T \partial F_{khl\nu}(0)\in\Z[\omega,\chi].
		\]
		In order to estimate the type of $P_{\partial khl\nu}(\omega,\chi)$, in view of Lemma~\ref{lemma: how to deal with type proved} it suffices for our purposes to bound the degree in $\omega$ and $\chi$ and the height of the above polynomial expression. First, $C(\omega)(3T)!c_2^T\beta_{\kappa,q}$ has bounded degree in $\omega$ and by \cite[Lemma 2]{Vasil'ev} logarithm of the height $\le c_6 T\log T$. Moreover,
		\[
		\log\left(\prod_{a=1}^g\binom{p_a+q_a}{p_a}\right)\le \log\left(2^{gT}\right)\le c_7T, \quad \log \left(\prod_{r=1}^n h_r^{p_r}\right)\le c_8T\log D_0,
		\]
		and these terms only contribute with their height, being constant in $\omega$. Furthermore, $C(\omega)^{nT}\prod_{r=1}^n \xi_r^{p_r}$ yields a polynomial expression in $\Z[\omega,\chi]$ with both degree in $\omega$ and $\chi$ satisfying $\le c_9 T$, while logarithm of the height $\le c_{10} T\log T$. For the product of the $\omega_{ij}$'s we have degree in $\omega$ and $\chi$ bounded by $ c_{11} D$ and logarithm of the height $\le c_{12} D\log D$, while for the product of the $e^{\xi_r\omega_{rj}}$'s we get degrees $\le c_{13}D_0 R$ and logarithm of the height $\le c_{14}D_0R\log (D_0 R)$. These computations lead to the desired inequality
		\[
		t(P_{\partial khl\nu}(\omega,\chi))\le c_4 D_0R\log (D_0R).
		\]
	\end{proof}

	We may now conclude the proof of Proposition~\ref{prop: aux. function 3}. We first consider the function
	\[
	\widetilde{\Phi}(z)\coloneqq C(N)\sum_{\|h\|\le D_0}\sum_{\|l\|\le D}\sum_{\|\nu\|\le D} \widetilde{E}_{hl\nu}(\omega,\chi) \prod_{r=1}^ne^{h_r\xi_r z_r}\prod_{i=1}^{m}H_{i}^{l_{i}}(z)
	\prod_{s=1}^g A_s^{\nu_s}(z),
	\]
	for some $\widetilde{E}_{hl\nu}(\omega,\chi)\in\Z[\omega,\chi]$ not all zero to be chosen in such a way that 
	\[
	\partial \widetilde{\Phi}(k\cdot \lambda)=0 \quad \text{for $|\partial|\le T$ and $k\in\Z^{2g}$ with $\|k\|\le R$}.
	\]
	We may regard these conditions as a linear system in the $\widetilde{E}_{hl\nu}$'s whose coefficients coincide with the $P_{\partial k h l \nu}(\omega,\chi)$ of Lemma~\ref{lemma: coeff for Siegel 3}. Since $T^gR^{2g}<D_0^{n}D^{g+m}$, Siegel's Lemma~\ref{lemma: Siegel} ensures the existence of the desired $\widetilde{E}_{hl\nu}$, and in combination with Lemma~\ref{lemma: coeff for Siegel 3} it also yields
	\[
	t(\widetilde{E}_{hl\nu}(\omega,\chi))\le c_{16} D_0R \log (D_0R).
	\]
 
	We are only left with removing $\chi$ from the coefficients $\widetilde{E}_{hl\nu}(\omega,\chi)$, which can be achieved via a standard argument. We define $\Phi(z)$ to be the product of all the functions $\widetilde{\Phi}_\sigma$ obtained from $\widetilde{\Phi}$ by replacing each $\widetilde{E}_{hl\nu}$ by $\sigma(\widetilde{E}_{hl\nu})$ for a Galois automorphism $\sigma$ of $\Q(\omega,\chi)$ over $\Q(\omega)$. Such $\Phi(z)$ finally satisfies the desired properties as in Proposition~\ref{prop: aux. function 3}.
		
	\section{End of the proof}
    We shall now construct a polynomial $Q$ with integer coefficients which realizes the transcendence measure for $\omega$ in the statement of Theorem~\ref{theo: 3}. This polynomial is a byproduct of the result of the previous section, and the missing estimate for $|Q(\omega)|$ will be obtained by analytic means.

	The functions $A_1,\dots, A_g$, $H_1,\dots, H_m$ and $e^{\xi_1z_1},\dots, e^{\xi_nz_n}$ are algebraically independent \cite[Corollary 7]{Brownawell-Kubota}, so $\Phi(z)$ does not vanish identically. As a result, we can find an integer $N_0\ge N$ such that 
	\[
	\partial \Phi(k\cdot\lambda)=0 \quad \text{for $\|k\|\le N_0^r,\; |\partial|\le N_0^t$},
	\]
	but there are $k_0$, $\partial_0$ with $\|k_0\|\le (N_0+1)^r$, $|\partial_0|<(N_0+1)^t$ such that
	\begin{gather*}
	\partial \Phi(k_0\cdot \lambda)=0 \quad \text{for $|\partial|<|\partial_0|$,}\\
	\partial_0\Phi(k_0\cdot\lambda)\neq 0.
	\end{gather*}
	We set for short $R_0=(N_0+1)^r$ and $T_0=(N_0+1)^t$. The same argument proposed in Lemma~\ref{lemma: coeff for Siegel 3} shows that 
	\[
	P(\omega,\chi)\coloneqq \frac{C(N_0)^\delta}{C(N)^\delta}\partial_0\Phi(k_0\cdot \lambda)
	\]
	is a non-zero element of $\Z[\omega,\chi]$ of type
	\[
	t(P(\omega,\chi))\le c_{17} D_0R_0\log(D_0R_0).
	\]

	We now aim at estimating $|P(\omega,\chi)|$ from above: by eventually removing $\chi$, this will lead us to the desired transcendence measure for $\omega$. The main tool that we are going to exploit is the Bombieri-Lang version of Schwarz's Lemma.
	\begin{proposition}
		\label{prop: Schwarz lemma several variables}
		There exist positive constants $c_{16}$, $c_{17}$ and $c_{18}$, only depending on $\lambda_1\dots, \lambda_{2g}$, such that for any $T\ge 1$, $\rho\ge 1$ and $\rho_1\ge c_{16}\rho$ and any entire function $f:\C^g\to \C$ satisfying the equation $\partial f(k\cdot\lambda)=0$ for $|\partial|<T$ and $\|k\|\le \rho$ the following inequality holds:
		\[
		|f|_{\rho_1}\le |f|_{c_{17}\rho_1}\exp(-c_{18}TR^2).
		\]
	\end{proposition}
	\begin{proof}
		\cite[Lemma 7]{Masser}
	\end{proof}

	We consider the entire function
	\[
	\Psi(z)=\theta_0(z)^{(g+m)\delta D}\Phi(z).
	\]
	Since $k_0\cdot \lambda$ is contained in the ball centred at the origin of radius $\le c_{18}R_0$, by Proposition~\ref{prop: Schwarz lemma several variables} we infer that
	\[
	|\Psi|_{c_{19}R_0}\le |\Psi|_{c_{20}R_0}e^{-c_{21}T_0 R_0^2}.
	\]
	The coefficients $E_{hl\nu}$ have modulus at most $e^{c_{22}D_0R\log(D_0R)}$ by Lemma~\ref{lemma: coeff for Siegel 3}. We use an elementary Lemma in order to obtain an upper bound for $|\Psi|_{c_{20}R_0}$.
	\begin{lemma}
		\label{lemma: order of growth theta}
		Let $G(z)$ be one of the functions $1,A_1,\dots, A_g, H_1,\dots, H_{g+1}$. Then $\theta G$ is entire and for any $\rho>0$ 
		\[
		|\theta(z)G(z)|_\rho\le e^{c_{15}\rho^2}
		\]
	\end{lemma}
	\begin{proof}
		By definition of theta function, for all $j=1,\dots,2g$ we may find $u_{j1},\dots, u_{jg}, v_j\in \C$ satisfying
		\[
		\theta(z+\lambda_j)=\theta(z)\exp( 2\pi i(u_{j1}z_1+\dots+u_{jg}z_g+v_j)).
		\]
		Let us denote by $U$ the $2g\times g$ complex matrix with entries the $u_{ij}$'s and by $V$ the vector $(v_1,\dots,v_{2g})\in\C^{2g}$. For any $k\in\Z^{2g}$ we have
		\[
		\theta(z+k\cdot \lambda)=\theta(z)\exp\left(2\pi i(\,^tk U z+ V\cdot k )\right).
		\]
		Let us now pick $z\in\C^g$ with $|z|\le\rho$. We may write $z=w+k\cdot\lambda$ for some $k\in\Z^g$ and $w$ in the fundamental parallelogram generated by $\lambda_1,\dots,\lambda_{2g}$. If $M$ denotes the maximum of $|\theta|$ on such parallelogram, we have
		\[
		|\theta(z)|=|\theta(w)|\exp\left(\text{Re}(2\pi i\,^tk U z+ 2\pi iV\cdot k )\right)\le M\exp\left(|2\pi\,^tk U z+ 2\pi V\cdot k|\right).
		\]
		Let $L$ be the minimum modulus of the $\lambda_j$'s, so that $\|k\|\le \frac{1}{L}\rho$. If $\|U\|$ denotes the maximum modulus of the $u_{ij}$'s, then 
		\[
		\left|^tk Uz\right|\le \frac{2g^2\|U\|}{L}\rho^2, \quad |V\cdot  k|\le 2g\|V\|\rho,
		\]
		and the claim for $G=1$ follows. The remaining cases can be treated analogously, by exploiting the fact that $G$ is either abelian or quasi-periodic.
	\end{proof}

	This computations lead us to the estimate
	\[
	|\Psi|_{c_{20}R_0}\le e^{c_{22}(D_0R_0\log(D_0R)+DR_0^2)},
	\]
	and therefore 
	\[
	|\Psi|_{c_{19}R_0}\le e^{-c_{23}T_0R_0^2}.
	\]
	A similar inequality carries through to $\partial_0\Psi(k_0\cdot\lambda)$. Indeed, by Cauchy's formula
	\[
	\frac{\partial^t\Psi(z)}{\partial z_1^{t_1}\dots \partial z_g^{t_g}}= \frac{t_1!\dots t_g!}{(2\pi i)^g}\int_{\gamma_1}\dots \int_{\gamma_g} \frac{\Psi(\zeta)}{(\zeta_1-z_1)^{t_1+1}\dots(\zeta_g-z_g)^{t_g+1}}\,d\zeta,
	\]
	where $\gamma_i$ denotes a circle centred at $z_i$ of radius $\le 1$ for all $i=1,\dots, g$. Hence, we derive the inequality
	\[
	\left|\partial_0 \Psi(k_0\cdot\lambda) \right|\le \exp(-c_{24}T R_0^2).
	\]
	Since $\theta_0(0)\neq 0$, it follows that $|\theta_0(k_0\cdot\lambda)^{(g+m)\delta D}|\ge e^{c_{25}DR_0^2}$, hence
	\[
	|P(\omega,\chi)|=\frac{|\partial_0\Psi(k_0\cdot\lambda)|}{|\theta_0(k_0\cdot\lambda)^{(g+m)\delta D}|}\le e^{-c_{26}T_0R_0^2}.
	\]
 
	By taking the norm of $P(\omega,\chi)$ over $\Q(\omega)$, we eventually obtain a polynomial $Q$ with integer coefficients satisfying
	\[
	0<|Q(\omega)|\le e^{-c_{27}T_0R_0^2}, \qquad t(Q)\le c_{28} D_0R_0\log(D_0R_0). 
	\]
	An exposition of the missing computation can be found for example in \cite{Feldman}. Since $(D_0R_0\log(D_0R_0))^{2+\frac{1}{a}}<T_0R_0^2$, we conclude that $\omega$ has the desired transcendence type.
	
	\section{Some applications}
	
	Let us now comment on some features of Theorem~\ref{theo: 3}. As we previously remarked, almost all transcendental numbers have transcendence type $\le 2+\epsilon$ for any $\epsilon>0$. It should therefore be expected that in almost all cases Theorem~\ref{theo: 3} yields in fact the existence of two algebraically independent numbers in $S$, provided $2m+n>2g$. Another reason why it seems likely to turn this Theorem into the form $\trdeg(\Q(S)/\Q)\ge 2$ is that it is possible to see that $\pi$ belongs to the field generated by the entries of $\Omega$. Thus, in case $\pi$ can actually be generated by fewer than $g+1$ rows of $\Omega$, we deduce the existence of two algebraically independent numbers among the ones in $S$ when selecting precisely those rows.
 
	The main obstruction to turning Theorem~\ref{theo: 3} in an algebraic independence statement lies in the fact that we have not been able to apply Gelfond's criterion \cite[Théorème 5.1.1]{Waldschimdt-fr} at the end of the proof. By letting $N$ range through all sufficiently large natural numbers, our argument yields a sequence of polynomials $Q_N$ satisfying 
	\[
	0<|Q_N(\omega)|\le e^{-c_{27}T_0R_0^2}, \qquad t(Q_N)\le c_{28} D_0R_0\log(D_0R_0),
	\]
	where the quantities $T_0,R_0, D_0$ are defined as above and depend on some $N_0\ge N$. If we managed to bound $N_0$ from above by a suitable power of $N$, Gelfond's criterion would lead to a contradiction with the transcendence of $\omega$. One could achieve this by supposing that the function $\Phi(z)$ of Proposition~\ref{prop: aux. function 3} vanishes with multiplicity $T_0$ at all points of the form $k\cdot\lambda$ with $\|k\|\le R_0$. Imposing these conditions yields a homogeneous linear system in the $E_{hl\nu}$, with the number of unknowns depending on $N$ and the one of equations depending on $N_0$. Using the fact that the $E_{hl\nu}$'s are not all zero, one may derive an upper bound of $N_0$ in terms of $N$ by computing the rank of the matrix associated with this linear system. However, the exponential terms in the expression of $\Phi(z)$ make it somewhat unclear how to practically compute such rank.

	We now examine some applications of Theorem~\ref{theo: 3}. For an integer $N\ge 3$, let us consider the curve $y^2=1-4x^N$, which has genus $\lfloor\frac{N-1}{2}\rfloor$. Let us also write $\zeta=e^{\frac{2\pi i}{N}}$. As shown in \cite[Chapter V]{Lang-introduction}, the period lattice of the Jacobian variety associated with this curve has generators given by the vectors
	\[
	\lambda_j=\left( \dots, \zeta^{kj}\left(1-\zeta^k\right)^2\frac{1}{N}B\left(\frac{k}{N},\frac{k}{N} \right), \dots\right)
	\] 
	for $j=0,\dots, N-1$, the components running over $k=1,\dots,\lfloor\frac{N-1}{2}\rfloor$, together with the vector
	\[
	\left( \dots, \left(1-\zeta^k\right)\frac{1}{N}B\left(\frac{k}{N},\frac{k}{N} \right), \dots\right).
	\] 
	An analogous expression applies to the quasi-periods, provided we let $k$ run from $\lfloor\frac{N+1}{2}\rfloor$ to $N-1$. 
 
	One may now apply Theorem~\ref{theo: 3} in order to derive results of algebraic independence for $B$-values, for instance by exploiting the fact that for any $a\in\Q\smallsetminus\Z$ the number $B(a,a)B(1-a,1-a)$ is a non-zero algebraic multiple of $\pi$. Let us go through some examples of these arguments. 
	\begin{corollary}
		For any non-zero complex number $\xi$, there are two algebraically independent numbers among 
		\[
		B\left( \frac{1}{12},\frac{1}{12}\right), \; B\left(\frac{5}{12},\frac{5}{12}\right),\; \pi,\; \xi, \; e^{\xi}, \; e^{i\sqrt{3}\xi}.
		\]
	\end{corollary}
	\begin{proof}
		We apply Theorem~\ref{theo: 3} to the complex Abelian variety described above for $N=12$ with the following choices. We choose the rows of the period matrix whose components are algebraic multiples of
		\[
		B\left( \frac{1}{12},\frac{1}{12}\right), \; B\left(\frac{5}{12},\frac{5}{12}\right),\; B\left( \frac{7}{12},\frac{7}{12}\right), \; B\left(\frac{11}{12},\frac{11}{12}\right).
		\]
		As for the matrix $E$ defined before Theorem~\ref{theo: 3}, we choose its fourth row, and we also pick 
		\[
		\xi_4=\xi\left(\left(1-\zeta^4\right)^2\frac{1}{12}B\left(\frac{4}{12},\frac{4}{12} \right)\right)^{-1},
		\]
		where $\zeta=e^{\frac{2\pi i}{12}}$. It is readily checked that the products
		\[
		B\left( \frac{1}{12},\frac{1}{12}\right)B\left(\frac{11}{12},\frac{11}{12}\right), \quad B\left(\frac{5}{12},\frac{5}{12}\right)B\left( \frac{7}{12},\frac{7}{12}\right)
		\]
		are non-zero algebraic multiples of $\pi$. Thus, $\pi$ belongs to the field generated by $S$, with $S$ defined as in Theorem~\ref{theo: 3}. As a result, $\trdeg(\Q(S)/\Q)\ge 2$. The numbers of $S$ appearing in the matrix $E$ are of the form $e^{\xi\zeta^{4j}}$ for some integers $j$, together with $e^{\xi(1-\zeta^{4})^{-1}}$. Moreover, $\zeta^4=\rho$, where $\rho=e^{\frac{2\pi i}{3}}$, while $(1-\zeta^4)$ is a primitive sixth root of unity. Hence, these numbers turn out to be
		\[
		e^\xi, \; e^{\xi\rho}, \; e^{\xi\rho^2}, \; e^{\xi\rho^{-1/2}}.
		\]
		Since $\rho=(-1+i\sqrt{3})/2$, it follows that the transcendence degree of $\Q(S)$ coincides with the one of the field generated over $\Q$ by the numbers in the statement, which is therefore proved.
	\end{proof}

	A similar strategy, choosing the third row of the matrix $E$, allows for example to prove that there are at least two algebraically independent numbers among
	\[
	B\left( \frac{1}{12},\frac{1}{12}\right), \; B\left(\frac{5}{12},\frac{5}{12}\right),\; \pi, \; \xi, \; e^{\xi},\; e^{i\xi}.
	\]
	By choosing $\xi=\log 2$ or $\xi=\pi^2$, one deduces for instance the existence of two algebraically independent numbers in each of the following sets:
	\begin{gather*}
    \left\{B\left( \frac{1}{12},\frac{1}{12}\right), \; B\left(\frac{5}{12},\frac{5}{12}\right),\; \pi, \; \log 2,\; 2^{i}\right\};\\
	\left\{B\left( \frac{1}{12},\frac{1}{12}\right), \; B\left(\frac{5}{12},\frac{5}{12}\right),\; \pi, \; e^{\pi^2},\; e^{i\pi^2}\right\}.
	\end{gather*}
	\newpage
	
	
\end{document}